\documentclass[11pt]{amsart}
\usepackage{amsmath}
\usepackage{amssymb}
\usepackage{amsfonts}

\newtheorem{theorem}{Theorem}[section]
\theoremstyle{plain}

\newtheorem{definition}[theorem]{Definition}

\newtheorem{lemma}[theorem]{Lemma}

\numberwithin{equation}{section}
\input{tcilatex}

\begin{document}
\title[Some bounds for the pseudocharacter of the space $C_{\alpha }\left(
X,Y\right) $]{Some bounds for the pseudocharacter of the space $C_{\alpha
}\left( X,Y\right) $}
\author{\c{C}etin Vural}
\address{Gazi Universitesi, Fen Fakultesi, Matematik Bolumu, 06500
Teknikokullar, Ankara, Turkey}
\email{cvural@gazi.edu.tr}
\date{}
\thanks{The author acknowledge Mr. Hasan G\"{u}l for the first draft of
manuscript.}

\begin{abstract}
Let $C_{\alpha }\left( X,Y\right) $ be the set of all continuous functions
from $X$ into $Y$ endowed with the set-open topology where $\alpha $ is a
hereditarily closed, compact network on $X.$ We obtain that:\bigskip

$i$-) $\psi \left( f,C_{\alpha }\left( X,Y\right) \right) \leq w\alpha
c\left( X\right) \cdot \underset{A\in \alpha }{\sup }\left( \psi \left(
f\left( A\right) ,Y\right) \right) \cdot \underset{A\in \alpha }{\sup }%
\left( w\left( f\left( A\right) \right) \right) \medskip $

$ii$-) $\psi \left( f,C_{\alpha }\left( X,Y\right) \right) \leq w\alpha
c\left( X\right) \cdot psw_{e}(f(X),Y).$
\end{abstract}

\maketitle

\section{Introduction and Terminology}

Let $X$ and $Y$ be topological spaces, and let $C\left( X,Y\right) $ denote
the set of all continuous mappings from $X$ into $Y.$ Let $\alpha $ be a
collection of subsets of $X.$ The topology having subbase $\left\{ \left[ A,V%
\right] :A\in \alpha \text{ and }V\text{ is an open subset of }Y\right\} $
on the set $C\left( X,Y\right) $ is denoted by $C_{\alpha }\left( X,Y\right) 
$ where \newline
$\left[ A,V\right] =\left\{ f\in C\left( X,Y\right) :f(A)\subseteq V\right\}
.$ If $\alpha $ consists of all finite subsets of $X,$ then the set $C\left(
X,Y\right) $ endowed with that topology is called pointwise convergence
topology and denoted by $C_{p}\left( X,Y\right) $.

The cardinality and the closure of a set is denoted by $\left\vert
A\right\vert $ and $cl\left( A\right) $, respectively. The restriction of a
mapping $f:X\rightarrow Y$ to a subset $A$ of $X$ is denoted by $f_{\mid
_{A}}.$ $T\left( X\right) $ denotes the set of all non-empty open subsets of
a topological space $X.$ $ord\left( x,\mathcal{A}\right) $ is the
cardinality of the collection $\left\{ A\in \mathcal{A}:x\in A\right\} $.
Throughout this paper $X$ and $Y$ are regular topological spaces, and $%
\alpha $ is a hereditarily closed, compact network on the domain space $X.$
(i.e., $\alpha $ is a network on $X$ such that each member of it is compact
and each closed subset of a member of it is a member of $\alpha .$) Without
loss of generality, we may assume that $\alpha $ is closed under finite
unions. Recall that \textit{the weak }$\alpha $-\textit{covering number }of $%
X$ is defined to be $w\alpha c\left( X\right) =\min \left\{ \left\vert \beta
\right\vert :\beta \subseteq \alpha \text{ and }\bigcup \beta \text{ is
dense in }X\right\} .$ The \textit{weight, density} and \textit{character}
of a space $X$ are denoted by $w(X),$ $d(X)$ and $\chi (X)$, respectively.
The \textit{i-weight} of a topological space $X,$ is the least of cardinals $%
w(Y)$ of the Tychonoff spaces $Y$ which are continuous one-to-one images of $%
X.$ \textit{The pseudocharacter} \textit{of a space} $X$ \textit{at a subset}
$A$, denoted by $\psi (A,X)$, is defined as the smallest cardinal number of
the form $\left\vert \mathcal{U}\right\vert ,$ where $\mathcal{U}$ is a
family of open subsets of $X$ such that $\bigcap \mathcal{U}=A$. If $%
A=\left\{ x\right\} $ is a singleton, then we write $\psi \left( x,X\right) $
instead of $\psi \left( \left\{ x\right\} ,X\right) .$ \textit{The
pseudocharacter} \textit{of a space} $X$ is defined to be $\psi \left(
X\right) =\sup \left\{ \psi \left( x,X\right) :x\in X\right\} $. \textit{The
diagonal number} $\Delta \left( X\right) $ of a space $X$ is the
pseudocharacter of its square $X\times X$ at its diagonal $\Delta
_{X}=\left\{ \left( x,x\right) :x\in X\right\} .$

The pseudocharacter of the space $C\left( X,Y\right) $ has been studied, and
some remarkable equalities or inequalities was obtained between the
pseudocharacter of the space $C\left( X,Y\right) $ for certain topologies
and some cardinal functions on the spaces $X$ and $Y.$ For instance, in \cite%
{Guthrie}, the inequalities $\psi \left( Y\right) \leq \psi \left(
C_{p}\left( X,Y\right) \right) \leq \psi \left( Y\right) \cdot d(X)$ and, in 
\cite{Arkhangelskii} and \cite{Arkhangelskii-2}, the equalities $\psi \left(
C_{p}\left( X,%
\mathbb{R}
\right) \right) =d(X)=iw\left( C_{p}\left( X,%
\mathbb{R}
\right) \right) $, and in \cite{McCoy-Ntantu}, $\psi \left( C_{\alpha
}\left( X,%
\mathbb{R}
\right) \right) =\Delta (C_{\alpha }\left( X,%
\mathbb{R}
\right) )=w\alpha c\left( X\right) $ were obtained. In this paper, when the
range space $Y$ is an arbitrary topological space instead of the space $%
\mathbb{R}
,$ we obtained some inequalities between the pseudocharacter of the space $%
C_{\alpha }\left( X,Y\right) $ at a point $f$ and the weak $\alpha $%
-covering number of the domain space $X$ and some cardinal functions on the
range space $Y.$

We assume that all cardinal invariants are at least the first infinite
cardinal $\aleph _{0}.$

Notations and terminology not explained above can be found in \cite{Kunen}
and \cite{Vaughan}.

\section{Main Results}

First, we give an inequality between the pseudocharacter of a point $f$ in
the space $C_{\alpha }\left( X,Y\right) $ and some cardinal functions on
spaces $X$ and $Y$.

\begin{theorem}
For each $f\in C_{\alpha }\left( X,Y\right) ,$ we have%
\begin{equation*}
\psi \left( f,C_{\alpha }\left( X,Y\right) \right) \leq w\alpha c\left(
X\right) \cdot \underset{A\in \alpha }{\sup }\left( \psi \left( f\left(
A\right) ,Y\right) \right) \cdot \underset{A\in \alpha }{\sup }\left(
w\left( f\left( A\right) \right) \right)
\end{equation*}%
\smallskip
\end{theorem}

\begin{proof}
Let $w\alpha c\left( X\right) \cdot \sup \left\{ \psi \left( f\left(
A\right) ,Y\right) :A\in \alpha \right\} \cdot \sup \left\{ w\left( f\left(
A\right) \right) :A\in \alpha \right\} =\kappa .$ The inequality $w\alpha
c\left( X\right) \leq \kappa $ gives us a subfamily $\beta =\left\{
A_{i}:i\in I\right\} $ of $\alpha $ such that $\left\vert I\right\vert \leq
\kappa $ and $X=cl\left( \bigcup \beta \right) =cl\left( \bigcup \left\{
A_{i}:i\in I\right\} \right) .$ Since $\psi \left( f\left( A_{i}\right)
,Y\right) \leq \kappa $ for each $i\in I,$ there exists a family $\mathcal{V}%
_{i}$ consisting of open subsets of the space $Y$ such that $\left\vert 
\mathcal{V}_{i}\right\vert \leq \kappa $ and $f\left( A_{i}\right) =\bigcap
\left\{ V:V\in \mathcal{V}_{i}\right\} $ for each $i\in I.$ Since $w\left(
f\left( A_{i}\right) \right) \leq \kappa $ for each $i\in I,$ the subspace
has a base $\mathcal{B}_{i}$ with $\left\vert \mathcal{B}_{i}\right\vert
\leq \kappa .$ For each $i\in I,$ let \newline
$\mathcal{H}_{i}=\left\{ \left[ A_{i}\cap f^{-1}\left( cl\left( G\right)
\right) ,Y\backslash cl\left( U\right) \right] :G,U\in \mathcal{B}_{i}\text{
and }cl\left( G\right) \cap cl\left( U\right) =\emptyset \right\} ,$\newline
$\mathcal{R}_{i}=\left\{ \left[ A_{i},V\right] :V\in \mathcal{V}_{i}\right\} 
$ and $\mathcal{W}=\left( \bigcup\nolimits_{i\in I}\mathcal{R}_{i}\right)
\cup \left( \bigcup\nolimits_{i\in I}\mathcal{H}_{i}\right) .$\newline
It is clear that $\left\vert \mathcal{W}\right\vert \leq \kappa $ and $f\in W
$ for each $W\in \mathcal{W}$. Now, we shall prove that $\bigcap \mathcal{W=}%
\left\{ f\right\} .$ Take a $g\in \bigcap \mathcal{W}.$ We claim that $%
g_{\mid A_{i}}=f_{\mid A_{i}}$ for each $i\in I.$ Assume the contrary.
Suppose $g_{\mid A_{j}}\neq f_{\mid A_{j}}$ for some $j\in I$ that is, we
have an $x\in A_{j}$ such that $f(x)\neq g(x).$ Since $g\in \bigcap \mathcal{%
W}$ and $f\left( A_{j}\right) =\bigcap \left\{ V:V\in \mathcal{V}%
_{j}\right\} ,$ we have $g(A_{j})\subseteq f\left( A_{j}\right) .$ Therefore 
$g(x)\in f\left( A_{j}\right) $ and $f(x)\in f\left( A_{j}\right) .$ Since $%
f(x)\neq g(x)$ and the space $Y$ is regular, there exist $G$ and $U$ in $%
\mathcal{B}_{j}$ such that $f(x)\in cl\left( G\right) ,$ $g(x)\in cl\left(
U\right) $ and $cl\left( G\right) \cap cl\left( U\right) =\emptyset .$ On
the other hand, since $\left[ A_{j}\cap f^{-1}\left( cl\left( G\right)
\right) ,Y\backslash cl\left( U\right) \right] \in \mathcal{H}_{j}$ and $%
g\in \bigcap \mathcal{W},$ we have $g\in \left[ A_{j}\cap f^{-1}\left(
cl\left( G\right) \right) ,Y\backslash cl\left( U\right) \right] .$ But this
contradicts to the fact that $g(x)\in cl\left( U\right) .$ Hence, $g_{\mid
A_{i}}=f_{\mid A_{i}}$ for each $i\in I,$ or in other words $g_{\mid
\bigcup\nolimits_{i\in I}A_{i}}=f_{\mid \bigcup\nolimits_{i\in I}A_{i}}.$
Hausdorffness of the space $Y$ and the equality $X=cl\left( \bigcup \beta
\right) =cl\left( \bigcup \left\{ A_{i}:i\in I\right\} \right) $ lead us to
the fact that $g=f.$ Therefore $\bigcap \mathcal{W=}\left\{ f\right\} $,
that is $\psi \left( f,C_{\alpha }\left( X,Y\right) \right) \leq \kappa .$
\end{proof}

Recall that a cover $\mathcal{A}$ of a set $X$ is called a \textit{%
separating cover} if \newline
$\bigcap \left\{ A\in \mathcal{A}:x\in A\right\} =\left\{ x\right\} ,$ for
each $x\in X.$ Also recall that \textit{the point separating weight} $%
psw\left( X\right) $ of a topological space $X$ is the smallest infinite
cardinal $\kappa $ such that the space $X$ has a separating open cover $%
\mathcal{V}$ with $ord(x,\mathcal{V})\leq \kappa $ for each $x\in X.$

\begin{definition}
Let $A$ be a subset of a topological space $\left( X,\tau \right) .$ We say
that the point separating exterior weight $psw_{e}\left( A,X\right) \leq
\kappa ,$ if there exists a subfamily $\mathcal{V}\subseteq \tau $
satisfying $ord(a,\mathcal{V})\leq \kappa $ and\textit{\ }$\bigcap \left\{
V\in \mathcal{V}:a\in V\right\} =\left\{ a\right\} ,$ for each $a\in A.$
\end{definition}

The following lemmas are needed for the second main theorem, and in order to
prove them, let us recall the Mi\v{s}\v{c}enko's lemma.

\begin{lemma}[\textbf{Mi\v{s}\v{c}enko's lemma \protect\cite{Vaughan}}]
Let $\kappa $ be an infinite cardinal, let $X$ be a set, and let $\mathcal{A}
$ be a collection of subsets of $X$ such that $ord\left( x,\mathcal{A}%
\right) \leq \kappa $ for all $x\in X.$ Then the number of finite minimal
covers of $X$ by elements of $\mathcal{A}$ is at most $\kappa .$
\end{lemma}

\begin{lemma}
Let $Z$ be subspace of the space $X$ such that $psw_{e}\left( Z,X\right)
\leq \kappa .$ Then $\psi \left( K,X\right) \leq \kappa $ for each compact
subset $K$ of $Z.$
\end{lemma}

\begin{proof}
Let $\mathcal{V}$ be a family of open subsets of $X$ satisfying $ord(z,%
\mathcal{V})\leq \kappa $ and\textit{\ }$\bigcap \left\{ V\in \mathcal{V}%
:z\in V\right\} =\left\{ z\right\} ,$ for each $z\in Z.$ Let $K$ be any
compact subspace of $Z$ and let\newline
$\mu =\left\{ \mathcal{W}:\mathcal{W\subseteq V}\text{ and }\mathcal{W}\text{
is a minimal finite open cover for }K\right\} .$ By Mi\v{s}\v{c}enko's
lemma, we have $\left\vert \mu \right\vert \leq \kappa .$

Define the family $\ \mathcal{O}=\left\{ \bigcup\nolimits_{W\in \mathcal{W}%
}W:\mathcal{W\in \mu }\right\} .$ It is clear that $\left\vert \mathcal{O}%
\right\vert \leq \kappa $ and it can be easily seen that $\bigcap \mathcal{O}%
=K.$ Hence $\psi \left( K,X\right) \leq \kappa .$
\end{proof}

\begin{lemma}
Let $Z$ be subspace of the space $X$ such that $psw_{e}\left( Z,X\right)
\leq \kappa .$ Then we have $w\left( K\right) \leq \kappa $ for each compact
subset $K$ of $Z.$
\end{lemma}

\begin{proof}
Let $K$ be a compact subset of $Z.$ Clearly, $psw(K)\leq psw(Z)\leq
psw_{e}(Z,X).$ The compactness of $K$ leads us to the fact that $w\left(
K\right) =psw(K).\left[ \text{in }\left[ 5\right] ,\text{ Ch. 1, Theorem 7.4}%
\right] .$ Hence the claim.
\end{proof}

Now, we are ready to give another bound for the pseudocharacter of the space 
$C_{\alpha }\left( X,Y\right) $ at a point $f.$

\begin{theorem}
For each $f\in C_{\alpha }\left( X,Y\right) ,$ we have%
\begin{equation*}
\psi \left( f,C_{\alpha }\left( X,Y\right) \right) \leq w\alpha c\left(
X\right) \cdot psw_{e}(f(X),Y).
\end{equation*}
\end{theorem}

\begin{proof}
Let $w\alpha c\left( X\right) \cdot psw_{e}(f(X),Y)=\kappa ,$ and let $\beta
=\left\{ A_{i}:i\in I\right\} $ be a subfamily of $\alpha $ such that $%
cl\left( \bigcup \beta \right) =X$ and $\left\vert I\right\vert \leq \kappa .
$ The compactness of $A_{i}$ for each $i\in I$ and the inequality $%
psw_{e}(f(X),Y)\leq \kappa $ lead us to the facts that $\psi \left(
f(A_{i}),Y\right) \leq \kappa $ and $w\left( A_{i}\right) \leq \kappa $ for
each $i\in I$, by lemmas.2.4 and 2.5 Therefore, by Theorem 2.1, we have $%
\psi \left( f,C_{\alpha }\left( X,Y\right) \right) \leq \kappa .$
\end{proof}

\end{document}